\documentclass[11pt]{amsart}

\usepackage[margin=1.08in]{geometry}
\usepackage{amsmath,amssymb,amsthm,mathtools}
\usepackage{booktabs}
\usepackage{placeins}
\usepackage{microtype}
\usepackage{xcolor}
\usepackage[colorlinks=true,linkcolor=blue!55!black,citecolor=blue!55!black,urlcolor=blue!60!black]{hyperref}
\hypersetup{
  pdftitle={Lyapunov Functions and R-Linear Convergence for Quadratic Barzilai--Borwein Dynamics},
  pdfauthor={Shutai Yang and Ya-xiang Yuan},
  pdfsubject={Barzilai--Borwein dynamics; Lyapunov functions; R-linear convergence; strongly convex quadratics},
  pdfkeywords={Barzilai--Borwein method, quadratic dynamics, Lyapunov function, R-linear convergence, invariant spectral faces}
}

\numberwithin{equation}{section}

\newtheorem{theorem}{Theorem}[section]
\newtheorem{proposition}[theorem]{Proposition}
\newtheorem{corollary}[theorem]{Corollary}
\newtheorem{lemma}[theorem]{Lemma}
\theoremstyle{definition}

\theoremstyle{remark}
\newtheorem{remark}[theorem]{Remark}

\newcommand{\R}{\mathbb{R}}
\newcommand{\norm}[1]{\lVert #1\rVert}
\newcommand{\spec}{\operatorname{spec}}
\newcommand{\diag}{\operatorname{diag}}
\newcommand{\BB}{\mathrm{BB}}

\title[Lyapunov Functions and R-Linear Convergence]{Lyapunov Functions and R-Linear Convergence for Quadratic Barzilai--Borwein Dynamics}

\author{Shutai Yang}
\address{State Key Laboratory of Scientific and Engineering Computing, Academy of Mathematics and Systems Science, Chinese Academy of Sciences, and University of Chinese Academy of Sciences, Beijing, China}
\address{School of Mathematical Sciences, University of Science and Technology of China, Hefei, China}
\email{yangshutai26@mails.ucas.ac.cn}

\author{Ya-xiang Yuan}
\address{State Key Laboratory of Scientific and Engineering Computing, Academy of Mathematics and Systems Science, Chinese Academy of Sciences, Beijing, China}
\email{yyx@lsec.cc.ac.cn}
\thanks{The work of Shutai Yang was supported by the National College Students Innovation and Entrepreneurship Training Program (Project No.~202510358091). The work of Ya-xiang Yuan was supported by the National Natural Science Foundation of China (Grant No.~12288201).}

\date{August 2026}

\begin{document}

\begin{abstract}
The Barzilai--Borwein (BB) method is $R$-linearly convergent on strongly convex quadratics, although neither the objective value nor the gradient norm is generally monotone. We construct an explicit current-point Lyapunov function for the quadratic dynamics after the warm-up step, assuming that both endpoints of the initial active spectrum survive this step. If $a<b$ are these endpoints and $P_a,P_b$ are the corresponding spectral projectors, then
\[
  \overline{f}(x)
  :=\lVert P_a\nabla f(x)\rVert^{\frac{2b}{a+b}}
    \lVert P_b\nabla f(x)\rVert^{\frac{2a}{a+b}}
\]
satisfies, for BB1, BB2, and every fixed positive weighted delayed Rayleigh rule,
\[
  \overline{f}(x_{k+1})
  =\left(\frac{b-a}{b+a}\right)^2
    e^{-2D_{a,b}(\tau_k)}\overline{f}(x_k),
  \qquad D_{a,b}(\tau_k)\ge0,
\]
where $\tau_k=1/\alpha_k$ is the reciprocal step. The function is obtained by denormalizing the endpoint coboundary certificate in the sharp-rate analysis of Yang and Yuan. We prove that the endpoint exponents are the unique minimax choice among endpoint monomials and illustrate the Lyapunov law on a nonmonotone BB1 trajectory. Using this law, we also give a proof of $R$-linear convergence of the method in finite dimensions.
\end{abstract}

\maketitle

\section{Introduction}

Consider the strongly convex quadratic problem
\begin{equation}\label{eq:quadratic}
  \min_{x\in\R^d} f(x)
  :=\frac12 x^{\mathsf T}Hx-b^{\mathsf T}x,
  \qquad H=H^{\mathsf T}\succ0.
\end{equation}
Let
\[
  x_*:=H^{-1}b,
  \qquad
  g(x):=\nabla f(x)=H(x-x_*).
\]
The Barzilai--Borwein method is globally convergent on \eqref{eq:quadratic}, but its individual steps need not decrease $f$; this nonmonotone behavior is part of the classical theory \cite{BarzilaiBorwein1988,Raydan1993,DaiLiao2002}.  The motivating question is therefore:

\begin{quote}
Can one construct from the quadratic objective a new function $\overline f$ such that the BB points generated for $f$ satisfy
\[
  \overline f(x_{k+1})\le \overline f(x_k)?
\]
\end{quote}

We answer this question by constructing an endpoint Lyapunov function for the dynamics after the warm-up step.  We make one convention throughout the article: the initial gradient is denoted by $g_{-1}$, its warm-up successor by $g_0$, and every update with index $k\ge0$ is a genuine delayed BB update.  Set
\begin{equation}\label{eq:post-warmup-spectrum-intro}
  \Lambda:=\Lambda(g_{-1}),
\end{equation}
and suppose that $|\Lambda|\ge2$ and that its two endpoints remain active in $g_0$.  We omit the one-point case because a genuine BB step then uses the reciprocal eigenvalue and reaches the minimizer immediately.

Write
\begin{equation}\label{eq:active-endpoints-intro}
  a:=\min\Lambda,
  \qquad
  b:=\max\Lambda,
  \qquad
  0<a<b,
\end{equation}
so the endpoint assumption is
\begin{equation}\label{eq:initial-endpoint-persistence}
  P_ag_j\ne0,
  \qquad
  P_bg_j\ne0
  \qquad(j=-1,0).
\end{equation}
Let $P_a,P_b$ be the corresponding orthogonal spectral projectors.  The component recurrence below shows that
\begin{equation}\label{eq:post-warmup-support}
  \Lambda(g_k)\subseteq\Lambda,
  \qquad
  P_ag_k\ne0,
  \qquad
  P_bg_k\ne0
  \qquad(k\ge-1).
\end{equation}
Thus the endpoints and their projectors are fixed, although interior spectral components may disappear.  The associated endpoint Lyapunov function is
\begin{equation}\label{eq:lyapunov-intro}
  \boxed{
  \overline f(x)
  :=\norm{P_ag(x)}^{\frac{2b}{a+b}}
    \norm{P_bg(x)}^{\frac{2a}{a+b}}.}
\end{equation}
It is two-homogeneous in the error.  Indeed,
\begin{equation}\label{eq:error-form-intro}
  \overline f(x)
  =a^{\frac{2b}{a+b}}b^{\frac{2a}{a+b}}
   \norm{P_a(x-x_*)}^{\frac{2b}{a+b}}
   \norm{P_b(x-x_*)}^{\frac{2a}{a+b}}.
\end{equation}
On the post-warm-up tail, both endpoint components are nonzero, so $\overline f(x_k)>0$.  $\overline f$ is therefore a Lyapunov function which obeys the exact one-step law
\begin{equation}\label{eq:intro-law}
  \overline f(x_{k+1})
  =c_{a,b}^{\,2}e^{-2D_{a,b}(\tau_k)}\overline f(x_k)
  \le c_{a,b}^{\,2}\overline f(x_k),
  \qquad
  c_{a,b}:=\frac{b-a}{b+a}<1,
\end{equation}
where $\tau_k=1/\alpha_k$ is the reciprocal BB step and $D_{a,b}\ge0$.  Thus $\overline f(x_k)$ is strictly decreasing even though neither $f(x_k)$ nor $\norm{g_k}$ need be monotone.

The construction was suggested by the endpoint coboundary identity in the sharp-rate analysis of Yang and Yuan \cite{YangYuan2026}.  In normalized spectral variables, their identity has the form
\[
  \log r=\log c-D+h\circ T-h.
\]
The key observation is that multiplying the unnormalized gradient norm by $e^{-h}$ produces exactly the square root of \eqref{eq:lyapunov-intro}.

The same identity also gives a proof of the classical $R$-linear convergence result.  If the full gradient norm does not contract, \eqref{eq:intro-law} forces the normalized state toward a boundary on which one endpoint energy vanishes.  Induction over the finite family of faces and compactness then give constants $C_\Lambda<\infty$ and $\rho_\Lambda\in(0,1)$ such that
\begin{equation}\label{eq:intro-rlinear}
  \norm{g_k}\le C_\Lambda\rho_\Lambda^k\norm{g_0}
  \qquad(k\ge0).
\end{equation}

This viewpoint is related to a longer spectral-dynamical tradition.  Akaike transformed normalized gradient energies into probability distributions \cite{Akaike1959}; Forsythe used a monotone normalized quantity to study asymptotic directions \cite{Forsythe1968}; and Pronzato, Wynn, and Zhigljavsky developed monotone moment quantities for a family of current-gradient methods \cite{PronzatoWynnZhigljavsky2006}.  Their family does not contain BB, whose delayed Rayleigh quotient requires a two-step state.  The endpoint coboundary in \cite{YangYuan2026} is the corresponding BB structure.

The remainder of this article records the post-warm-up dynamics, proves the exact Lyapunov law and the minimax optimality of its exponents, derives the function from the coboundary certificate, gives the invariant-face proof of $R$-linear convergence, and gives a numerical example.

\section{Post-warm-up quadratic BB dynamics}\label{sec:bb-dynamics}

Let the distinct eigenvalues of $H$ be
\[
  0<\lambda_1<\cdots<\lambda_s,
\]
and let $P_\lambda$ denote the orthogonal projector onto the eigenspace associated with $\lambda$.  Then
\[
  H=\sum_{\lambda\in\spec(H)}\lambda P_\lambda,
  \qquad
  g=\sum_{\lambda\in\spec(H)}P_\lambda g.
\]
For a nonzero vector $v$, define
\begin{equation}\label{eq:active-spectrum}
  \Lambda(v):=\{\lambda\in\spec(H):P_\lambda v\ne0\}.
\end{equation}

The post-warm-up BB tail satisfies
\begin{equation}\label{eq:bb-update}
  x_{k+1}=x_k-\alpha_kg_k,
  \qquad
  g_k:=g(x_k),
  \qquad k\ge0,
\end{equation}
so that
\begin{equation}\label{eq:gradient-update}
  g_{k+1}=(I-\alpha_kH)g_k.
\end{equation}
Every displayed step is generated from the preceding gradient.  Hence, for $k\ge0$,
\begin{equation}\label{eq:delayed-bb}
  \alpha_k^{\BB1}
  =\frac{g_{k-1}^{\mathsf T}g_{k-1}}
         {g_{k-1}^{\mathsf T}Hg_{k-1}},
  \qquad
  \alpha_k^{\BB2}
  =\frac{g_{k-1}^{\mathsf T}Hg_{k-1}}
         {g_{k-1}^{\mathsf T}H^2g_{k-1}}.
\end{equation}
Set
\begin{equation}\label{eq:tau}
  \tau_k:=\alpha_k^{-1}.
\end{equation}

\begin{lemma}\label{lem:tau-range}
Under the post-warm-up convention, the endpoint components remain nonzero and both BB rules satisfy
\begin{equation}\label{eq:tau-range}
  a<\tau_k<b
  \qquad(k\ge0).
\end{equation}
More generally, $\tau_k$ is a convex combination of the eigenvalues active in $g_{k-1}$.
\end{lemma}

\begin{proof}
For BB1,
\begin{equation}\label{eq:tau-bb1}
  \tau_k^{\BB1}
  =\frac{g_{k-1}^{\mathsf T}Hg_{k-1}}
         {g_{k-1}^{\mathsf T}g_{k-1}}
  =\sum_{\lambda\in\Lambda}
    \lambda\frac{\norm{P_\lambda g_{k-1}}^2}{\norm{g_{k-1}}^2}.
\end{equation}
For BB2,
\begin{equation}\label{eq:tau-bb2}
  \tau_k^{\BB2}
  =\frac{g_{k-1}^{\mathsf T}H^2g_{k-1}}
         {g_{k-1}^{\mathsf T}Hg_{k-1}}
  =\sum_{\lambda\in\Lambda}
    \lambda\frac{\lambda\norm{P_\lambda g_{k-1}}^2}
    {g_{k-1}^{\mathsf T}Hg_{k-1}}.
\end{equation}
The weights are nonnegative and sum to one.  At $k=0$, the weights at $a$ and $b$ are positive by \eqref{eq:initial-endpoint-persistence}, so the convex combination is strict.  Since the endpoint components of $g_0$ are also nonzero, \eqref{eq:gradient-update} gives
\[
  P_ag_1=\left(1-\frac{a}{\tau_0}\right)P_ag_0\ne0,
  \qquad
  P_bg_1=\left(1-\frac{b}{\tau_0}\right)P_bg_0\ne0.
\]
Repeating the same argument proves by induction that the endpoint components remain nonzero and that \eqref{eq:tau-range} holds for every $k\ge0$.  
Projecting \eqref{eq:gradient-update} onto an eigenspace gives
\begin{equation}\label{eq:component-update}
  P_\lambda g_{k+1}
  =\left(1-\frac{\lambda}{\tau_k}\right)P_\lambda g_k.
\end{equation}
The same component recurrence also shows that no spectral component outside $\Lambda$ can be created, proving \eqref{eq:post-warmup-support}.
\end{proof}

\begin{remark}\label{rem:weighted}
The same result holds for
\begin{equation}\label{eq:weighted-rule}
  \alpha_k^\Psi
  =\frac{g_{k-1}^{\mathsf T}\Psi(H)g_{k-1}}
         {g_{k-1}^{\mathsf T}\Psi(H)Hg_{k-1}},
  \qquad
  \Psi(\lambda)>0\quad(\lambda\in\Lambda).
\end{equation}
Indeed,
\[
  (\alpha_k^\Psi)^{-1}
  =\sum_{\lambda\in\Lambda}
   \lambda
   \frac{\Psi(\lambda)\norm{P_\lambda g_{k-1}}^2}
        {\sum_{\mu\in\Lambda}\Psi(\mu)\norm{P_\mu g_{k-1}}^2}
  \in(a,b).
\]
The choices $\Psi=I$ and $\Psi=H$ give BB1 and BB2, respectively.
\end{remark}

\section{An endpoint Lyapunov function and its exact dissipation law}\label{sec:lyapunov}

The elementary inequality below determines both the exponents and the contraction factor.

\begin{lemma}\label{lem:endpoint-ineq}
Let $0<a<b$ and $u\in[a,b]$.  Then
\begin{equation}\label{eq:endpoint-ineq}
  \left(\frac{u-a}{u}\right)^{\frac{b}{a+b}}
  \left(\frac{b-u}{u}\right)^{\frac{a}{a+b}}
  \le \frac{b-a}{b+a}.
\end{equation}
Equality holds if and only if $u=(a+b)/2$.
\end{lemma}

\begin{proof}
The endpoint cases are immediate.  For $u\in(a,b)$, put
\[
  t:=\frac{u-a}{b-u}>0.
\]
Then $u=(a+tb)/(1+t)$, and the left-hand side of \eqref{eq:endpoint-ineq} is
\[
  \frac{b-a}{a+tb}\,t^{\frac{b}{a+b}}.
\]
Weighted AM--GM applied to $1$ and $t$, with weights $a/(a+b)$ and $b/(a+b)$, gives
\[
  \frac{a+tb}{a+b}\ge t^{\frac{b}{a+b}}.
\]
This proves the inequality.  Equality holds exactly when $t=1$, equivalently $u=(a+b)/2$.
\end{proof}

Define
\begin{equation}\label{eq:lyapunov}
  \overline f(x)
  :=\norm{P_ag(x)}^{\frac{2b}{a+b}}
    \norm{P_bg(x)}^{\frac{2a}{a+b}}
\end{equation}
and
\begin{equation}\label{eq:cab}
  c_{a,b}:=\frac{b-a}{b+a}.
\end{equation}
For $u\in(a,b)$, define the nonnegative defect
\begin{equation}\label{eq:defect}
  D_{a,b}(u)
  :=\log c_{a,b}
    -\frac{b}{a+b}\log\frac{u-a}{u}
    -\frac{a}{a+b}\log\frac{b-u}{u}.
\end{equation}
Lemma~\ref{lem:endpoint-ineq} gives $D_{a,b}(u)\ge0$, with equality exactly at $u=(a+b)/2$.

\begin{theorem}\label{thm:exact-dissipation}
For every BB1 or BB2 step with $k\ge0$ on the post-warm-up tail,
\begin{align}
  \overline f(x_{k+1})
  &=
  \left(\frac{\tau_k-a}{\tau_k}\right)^{\frac{2b}{a+b}}
  \left(\frac{b-\tau_k}{\tau_k}\right)^{\frac{2a}{a+b}}
  \overline f(x_k) \label{eq:exact-factor}\\
  &=c_{a,b}^{\,2}e^{-2D_{a,b}(\tau_k)}\overline f(x_k).\label{eq:defect-identity}
\end{align}
Consequently,
\begin{equation}\label{eq:lyapunov-contraction}
  0<\overline f(x_{k+1})
  \le c_{a,b}^{\,2}\overline f(x_k)
  <\overline f(x_k)
  \qquad(k\ge0).
\end{equation}
The same statements hold for every fixed positive weighted delayed rule \eqref{eq:weighted-rule}.
\end{theorem}

\begin{proof}
By \eqref{eq:component-update} and $a<\tau_k<b$,
\[
  \norm{P_ag_{k+1}}
  =\frac{\tau_k-a}{\tau_k}\norm{P_ag_k},
  \qquad
  \norm{P_bg_{k+1}}
  =\frac{b-\tau_k}{\tau_k}\norm{P_bg_k}.
\]
Substitution into \eqref{eq:lyapunov} gives \eqref{eq:exact-factor}.  The definition \eqref{eq:defect} gives \eqref{eq:defect-identity}, and Lemma~\ref{lem:endpoint-ineq} gives \eqref{eq:lyapunov-contraction}.  The weighted case follows from Remark~\ref{rem:weighted}.
\end{proof}

The theorem therefore answers the original monotonicity question.

\begin{corollary}\label{cor:cumulative}
For every $k\ge1$,
\begin{equation}\label{eq:cumulative}
  \overline f(x_k)
  =c_{a,b}^{\,2k}
   \exp\!\left(-2\sum_{j=0}^{k-1}D_{a,b}(\tau_j)\right)
   \overline f(x_0).
\end{equation}
Equivalently, $c_{a,b}^{-2k}\overline f(x_k)$ is nonincreasing.
\end{corollary}

\begin{proof}
Multiply \eqref{eq:defect-identity} over $j=0,\ldots,k-1$.
\end{proof}

Formula \eqref{eq:cumulative} separates the worst endpoint rate from the additional dissipation.  If
\[
  \liminf_{k\to\infty}
  \frac1k\sum_{j=0}^{k-1}D_{a,b}(\tau_j)>0,
\]
then the root factor of this Lyapunov function is strictly smaller than $c_{a,b}^2$.  Conversely, a tail that is asymptotically extremal for this energy must have vanishing average defect, which forces most reciprocal steps to spend most of their time near the midpoint $(a+b)/2$.

\section{Why the exponents are optimal}\label{sec:optimality}

Consider the endpoint-monomial family of candidates
\begin{equation}\label{eq:Vp}
  V_p(x)
  :=\norm{P_ag(x)}^{2p}\norm{P_bg(x)}^{2(1-p)},
  \qquad 0<p<1.
\end{equation}
For a reciprocal step $u\in[a,b]$, the square-root one-step factor is
\begin{equation}\label{eq:Rp}
  R_p(u)
  :=\left(\frac{u-a}{u}\right)^p
    \left(\frac{b-u}{u}\right)^{1-p}.
\end{equation}

\begin{proposition}\label{prop:optimal-weights}
For every $0<p<1$,
\begin{equation}\label{eq:Gamma-p}
  \max_{u\in[a,b]}R_p(u)
  =(b-a)\frac{p^p(1-p)^{1-p}}{b^pa^{1-p}}.
\end{equation}
The unique minimizer of this worst one-step factor is
\begin{equation}\label{eq:pstar}
  p_*:=\frac{b}{a+b},
  \qquad
  1-p_*:=\frac{a}{a+b},
\end{equation}
and
\begin{equation}\label{eq:minimax-value}
  \min_{0<p<1}\max_{u\in[a,b]}R_p(u)
  =\frac{b-a}{b+a}=c_{a,b}.
\end{equation}
Thus \eqref{eq:lyapunov} is the unique minimax member of the family \eqref{eq:Vp}.
\end{proposition}

\begin{proof}
Differentiating $\log R_p(u)$ shows that its maximum is attained at
\[
  u_p^*=\frac{ab}{b(1-p)+ap}.
\]
At this point,
\[
  \frac{u_p^*-a}{u_p^*}=\frac{(b-a)p}{b},
  \qquad
  \frac{b-u_p^*}{u_p^*}=\frac{(b-a)(1-p)}{a},
\]
which proves \eqref{eq:Gamma-p}.  If the right-hand side is denoted by $\Gamma(p)$, then
\[
  \frac{d}{dp}\log\Gamma(p)
  =\log\frac{ap}{b(1-p)}.
\]
The unique critical point is $p=b/(a+b)$.  Strict convexity of $\log\Gamma$ gives uniqueness, and substitution yields \eqref{eq:minimax-value}.
\end{proof}

Thus, starting from a weighted geometric mean of the endpoint components and minimizing its worst admissible one-step factor recovers both the same exponents and the same constant $c_{a,b}$.

\section{The Lyapunov function hidden in the sharp-rate coboundary}\label{sec:derivation}

We now explain how \eqref{eq:lyapunov} can be read directly from the endpoint certificate in \cite{YangYuan2026}.  For a nonzero gradient $g_k$, define its normalized spectral energies
\begin{equation}\label{eq:wk}
  w_{k,\lambda}
  :=\frac{\norm{P_\lambda g_k}^2}{\norm{g_k}^2},
  \qquad \lambda\in\Lambda.
\end{equation}
For BB1, the natural normalized state is
\[
  \chi_k:=(w_k,w_{k-1}).
\]
Let $T$ be the induced two-step map and put
\[
  r(\chi_k):=\frac{\norm{g_{k+1}}}{\norm{g_k}}.
\]
The endpoint transfer function used in \cite{YangYuan2026} is
\begin{equation}\label{eq:h-transfer}
  h(\chi)
  =-\frac{b}{2(a+b)}\log w_a
   -\frac{a}{2(a+b)}\log w_b.
\end{equation}
Their coboundary identity has the form
\begin{equation}\label{eq:coboundary}
  \log r(\chi)
  =\log c_{a,b}-D(\chi)+h(T\chi)-h(\chi),
  \qquad D(\chi)\ge0.
\end{equation}
For the present tail, $D(\chi_k)=D_{a,b}(\tau_k)$.

The transfer term is exactly what is needed to de-normalize the state.  From \eqref{eq:wk} and \eqref{eq:h-transfer},
\begin{align}\label{eq:denormalization}
  \norm{g_k}e^{-h(\chi_k)}
  &=\norm{g_k}
    w_{k,a}^{\frac{b}{2(a+b)}}
    w_{k,b}^{\frac{a}{2(a+b)}}\\
  &=\norm{P_ag_k}^{\frac{b}{a+b}}
    \norm{P_bg_k}^{\frac{a}{a+b}}
   =\sqrt{\overline f(x_k)}.
\end{align}
Combining \eqref{eq:coboundary} with \eqref{eq:denormalization} gives
\begin{equation}\label{eq:denormalized-coboundary}
  \log\frac{\sqrt{\overline f(x_{k+1})}}
                 {\sqrt{\overline f(x_k)}}
  =\log c_{a,b}-D_{a,b}(\tau_k),
\end{equation}
which is precisely \eqref{eq:defect-identity} after squaring.

\section{R-linear convergence by invariant-face induction}\label{sec:rlinear}

We now use the Lyapunov law to prove $R$-linear convergence.  The argument is given first for BB1.  Fixed positive weighted delayed rules, including BB2, follow by spectral conjugacy.

Let the post-warm-up spectral set be
\[
  a=\lambda_1<\lambda_2<\cdots<\lambda_s=b,
  \qquad I:=\{1,\ldots,s\}.
\]
Set
\[
  \Delta_I:=\left\{w\in\R^I:w_i\ge0,\ \sum_{i\in I}w_i=1\right\},
  \qquad X_I:=\Delta_I\times\Delta_I.
\]
For $z\in\Delta_I$, define
\begin{equation}\label{eq:normalized-u-phi}
  u(z):=\sum_{i\in I}\lambda_i z_i,
  \qquad
  \phi_i(z):=\left(1-\frac{\lambda_i}{u(z)}\right)^2.
\end{equation}
For $\chi=(w,z)\in X_I$, let
\begin{equation}\label{eq:normalized-r}
  r_I(\chi)^2:=\sum_{i\in I}\phi_i(z)w_i.
\end{equation}
On the regular set $X_I^{\rm reg}:=\{\chi:r_I(\chi)>0\}$, put
\begin{equation}\label{eq:normalized-T}
  T_I(w,z)
  :=\left(\frac{\phi(z)\odot w}{r_I(w,z)^2},w\right),
\end{equation}
where $\odot$ denotes componentwise multiplication.  For an actual BB1 tail,
\begin{equation}\label{eq:normalized-orbit}
  \chi_k=(w_k,w_{k-1}),
  \qquad
  \chi_{k+1}=T_I\chi_k,
  \qquad
  \norm{g_{k+1}}=r_I(\chi_k)\norm{g_k}.
\end{equation}
If $r_I(\chi)=0$, the corresponding gradient update terminates.  For a nonempty $J\subset I$, identify $X_J=\Delta_J\times\Delta_J$ with the corresponding closed face of $X_I$.  Then $X_J$ is forward invariant, and the restrictions of $r_I$ and $T_I$ to $X_J$ are $r_J$ and $T_J$.

For $n\ge0$, define the terminally extended $n$-step multiplier by
\begin{equation}\label{eq:Rn-definition}
  \mathcal R_0(\chi):=1,
  \qquad
  \mathcal R_{n+1}(\chi):=
  \begin{cases}
    r_I(\chi)\mathcal R_n(T_I\chi),&r_I(\chi)>0,\\
    0,&r_I(\chi)=0.
  \end{cases}
\end{equation}
Thus $\mathcal R_n(\chi_k)=\norm{g_{k+n}}/\norm{g_k}$ on every nonterminal orbit segment, and the usual cocycle identity holds whenever the intermediate state is defined.

\begin{lemma}\label{lem:Rn-continuous}
For every $n\ge0$, the function $\mathcal R_n:X_I\to[0,\infty)$ is continuous.
\end{lemma}

\begin{proof}
The assertion is clear for $n=0$.  Suppose it holds for $n$, and set
\[
  M_I:=\max\left\{1,\max_{\chi\in X_I}r_I(\chi)\right\}<\infty.
\]
At a regular state, continuity follows from \eqref{eq:Rn-definition} and the continuity of $r_I$ and $T_I$.  If $r_I(\chi)=0$ and $\chi_j\to\chi$, then
\[
  0\le \mathcal R_{n+1}(\chi_j)
  \le M_I^n r_I(\chi_j)\longrightarrow0
  =\mathcal R_{n+1}(\chi).
\]
This proves the induction step.
\end{proof}

Let
\begin{equation}\label{eq:psi-I}
  \psi_I(w):=w_1^{\frac{b}{a+b}}w_s^{\frac{a}{a+b}},
  \qquad
  c_I:=\frac{b-a}{b+a}.
\end{equation}
If $T_I(w,z)=(w^+,w)$, then Lemma~\ref{lem:endpoint-ineq} gives
\begin{align}
  r_I(w,z)^2\psi_I(w^+)
  &=
  \left|1-\frac{a}{u(z)}\right|^{\frac{2b}{a+b}}
  \left|1-\frac{b}{u(z)}\right|^{\frac{2a}{a+b}}
  \psi_I(w)\notag\\
  &\le c_I^2\psi_I(w).
  \label{eq:normalized-lyapunov}
\end{align}
Consequently, along every nonterminal segment $\chi_n=(w_n,z_n)=T_I^n\chi$,
\begin{equation}\label{eq:normalized-lyapunov-n}
  \mathcal R_n(\chi)^2\psi_I(w_n)
  \le c_I^{2n}\psi_I(w_0).
\end{equation}

\begin{proposition}\label{prop:block-contraction}
For every nonempty finite index set $I$, there exist $L_I\ge1$ and $q_I\in(0,1)$ such that
\begin{equation}\label{eq:block-contraction}
  \min_{1\le\ell\le L_I}\mathcal R_\ell(\chi)
  \le q_I
  \qquad(\chi\in X_I).
\end{equation}
\end{proposition}

\begin{proof}
We argue by induction on $|I|$.  If $|I|=1$, the only state is terminal, and one may take $L_I=1$ and $q_I=1/2$.

Assume the result for every nonempty proper subset of $I$.  Let $M_I$ be as in Lemma~\ref{lem:Rn-continuous}, and set
\[
  \overline L:=\max_{\varnothing\ne J\subsetneq I}L_J,
  \qquad
  \overline q:=\max_{\varnothing\ne J\subsetneq I}q_J<1.
\]
Choose $m\ge1$ such that
\begin{equation}\label{eq:m-choice}
  M_I\overline q^{\,m}<\frac12.
\end{equation}
Consider the endpoint boundary
\begin{equation}\label{eq:BI}
  B_I:=\{(w,z)\in X_I:w_1w_s=0\}
      =\{(w,z):\psi_I(w)=0\}.
\end{equation}
Fix $\chi=(w,z)\in B_I$.  If $r_I(\chi)=0$, then $\mathcal R_1(\chi)=0$.  Otherwise, if $T_I\chi=(w^+,w)$ and $J:=\{i:w_i>0\}$, then $J\subsetneq I$ and both $w^+$ and $w$ are supported on $J$.  Hence $T_I\chi\in X_J$.  The face $X_J$ is forward invariant.  Applying the induction hypothesis successively $m$ times and using the cocycle property gives an index
\[
  1\le\ell\le 1+m\overline L
\]
such that
\[
  \mathcal R_\ell(\chi)
  \le M_I\overline q^{\,m}
  <\frac12.
\]
Thus, with $L_B:=1+m\overline L$, the continuous function
\[
  F_B(\chi):=\min_{1\le\ell\le L_B}\mathcal R_\ell(\chi)
\]
satisfies $F_B<1/2$ on $B_I$.  Therefore
\begin{equation}\label{eq:UB}
  U_B:=\{\chi\in X_I:F_B(\chi)<1/2\}
\end{equation}
is an open neighborhood of $B_I$.

We next show that every $\chi\in X_I$ has a contracting finite prefix.  Suppose to the contrary that
\begin{equation}\label{eq:no-prefix-contraction}
  \mathcal R_n(\chi)\ge1
  \qquad(n\ge1).
\end{equation}
Then the orbit never terminates.  Writing $T_I^n\chi=(w_n,z_n)$, \eqref{eq:normalized-lyapunov-n} and \eqref{eq:no-prefix-contraction} imply
\begin{equation}\label{eq:psi-to-zero}
  \psi_I(w_n)\le c_I^{2n}\psi_I(w_0)\longrightarrow0.
\end{equation}
The complement $X_I\setminus U_B$ is compact and disjoint from the zero set $B_I$.  If it is nonempty, $\psi_I$ has a positive minimum there.  Hence \eqref{eq:psi-to-zero} gives $T_I^n\chi\in U_B$ for all sufficiently large $n$.  Choose such an index $t_0$.  Given $t_j$, select $1\le\ell_j\le L_B$ with
\[
  \mathcal R_{\ell_j}(T_I^{t_j}\chi)<\frac12,
  \qquad t_{j+1}:=t_j+\ell_j.
\]
All selected states remain in $U_B$, and thus
\[
  \mathcal R_{t_j}(\chi)
  \le \mathcal R_{t_0}(\chi)\left(\frac12\right)^j.
\]
For large $j$ this is smaller than one, contradicting \eqref{eq:no-prefix-contraction}.

Hence the open sets $\{\chi:\mathcal R_n(\chi)<1\}$, $n\ge1$, cover the compact space $X_I$.  A finite subcover gives an integer $L_I$ for which
\[
  F_I(\chi):=\min_{1\le\ell\le L_I}\mathcal R_\ell(\chi)<1
  \qquad(\chi\in X_I).
\]
By continuity and compactness, $\max_{X_I}F_I<1$.  Replacing this maximum by any larger number in $(0,1)$ gives $q_I$ and proves \eqref{eq:block-contraction}.
\end{proof}

\begin{theorem}\label{thm:rlinear}
Let $\{g_k\}_{k\ge-1}$ be a BB1 gradient tail satisfying the post-warm-up convention with finite spectral set $\Lambda$.  Then there exist constants $C_\Lambda<\infty$ and $\rho_\Lambda\in(0,1)$ such that
\begin{equation}\label{eq:rlinear-gradient}
  \norm{g_k}\le C_\Lambda\rho_\Lambda^k\norm{g_0}
  \qquad(k\ge0).
\end{equation}
Consequently, the errors converge $R$-linearly with factor $\rho_\Lambda$, and the objective gaps converge $R$-linearly with factor $\rho_\Lambda^2$.
\end{theorem}

\begin{proof}
Let $L:=L_I$, $q:=q_I$, and $M:=M_I$.  Starting from $t_0=0$, apply Proposition~\ref{prop:block-contraction} at $\chi_{t_j}$ and choose $1\le\ell_j\le L$ such that
\[
  \mathcal R_{\ell_j}(\chi_{t_j})\le q,
  \qquad t_{j+1}:=t_j+\ell_j.
\]
If the method terminates, the conclusion is immediate.  Otherwise,
\begin{equation}\label{eq:block-endpoint-decay}
  \norm{g_{t_j}}\le q^j\norm{g_0}.
\end{equation}
For $t_j\le k<t_{j+1}$, at most $L$ one-step multipliers occur, and hence
\[
  \norm{g_k}\le M^Lq^j\norm{g_0}.
\]
Since $t_{j+1}\le(j+1)L$ and $k<t_{j+1}$, we have $j>k/L-1$.  With
\begin{equation}\label{eq:rho-C}
  \rho_\Lambda:=q^{1/L},
  \qquad
  C_\Lambda:=M^Lq^{-1},
\end{equation}
this proves \eqref{eq:rlinear-gradient}.

Finally, on the $\Lambda$-spectral subspace,
\[
  a\norm{x_k-x_*}\le\norm{g_k}\le b\norm{x_k-x_*}
\]
and
\[
  \frac{1}{2b}\norm{g_k}^2
  \le f(x_k)-f(x_*)
  \le\frac{1}{2a}\norm{g_k}^2.
\]
The error and objective-gap conclusions follow.
\end{proof}

\begin{corollary}\label{cor:rlinear-weighted}
The conclusion of Theorem~\ref{thm:rlinear} holds for every fixed positive weighted delayed Rayleigh rule \eqref{eq:weighted-rule}, and in particular for BB2.
\end{corollary}

\begin{proof}
On the $\Lambda$-spectral subspace, set $S:=\Psi(H)^{1/2}$ and $\widetilde g_k:=Sg_k$.  Since $S$ commutes with $H$ and is invertible there,
\[
  \widetilde g_{k+1}=(I-\alpha_k^\Psi H)\widetilde g_k,
  \qquad
  \alpha_k^\Psi
  =\frac{\norm{\widetilde g_{k-1}}^2}
         {\widetilde g_{k-1}^{\mathsf T}H\widetilde g_{k-1}}.
\]
Thus $\{\widetilde g_k\}$ is a BB1 gradient sequence with the same spectral supports at every index.  Theorem~\ref{thm:rlinear} applies to $\widetilde g_k$, and the fixed norm equivalence induced by $S$ transfers the estimate to $g_k$.  The choice $\Psi(\lambda)=\lambda$ gives BB2.
\end{proof}

\section{A two-dimensional BB1 trajectory}\label{sec:numerics}

We show an elementary example
\begin{equation}\label{eq:example-H}
  H=\diag(1,10),
  \qquad b=0.
\end{equation}
To place the computation directly in the post-warm-up convention, take the state pair
\begin{equation}\label{eq:example-state}
  g_{-1}=(1,2)^{\mathsf T},
  \qquad
  \widehat\alpha=\frac{5}{41},
  \qquad
  g_0=(I-\widehat\alpha H)g_{-1}
      =\left(\frac{36}{41},-\frac{18}{41}\right)^{\mathsf T}.
\end{equation}
Since $x_j=H^{-1}g_j$ and
\[
  x_0-x_{-1}=-\widehat\alpha g_{-1},
\]
this is a valid secant-generated BB state.  Starting at $k=0$, every displayed step is the genuine delayed BB1 update
\begin{equation}\label{eq:example-recurrence}
  \alpha_k
  =\frac{\norm{g_{k-1}}^2}{g_{k-1}^{\mathsf T}Hg_{k-1}},
  \qquad
  g_{k+1}=(I-\alpha_kH)g_k.
\end{equation}
The active endpoints are $a=1$ and $b=10$, so the endpoint Lyapunov function and its universal contraction bound are
\begin{equation}\label{eq:example-lyapunov}
  V_k:=\overline f(x_k)
  =|g_{k,1}|^{20/11}|g_{k,2}|^{2/11},
  \qquad
  c_{1,10}^2=\left(\frac9{11}\right)^2
  =\frac{81}{121}\approx0.669421.
\end{equation}
For comparison, write
\[
  F_k:=f(x_k)-f(x_*)
  =\frac12g_k^{\mathsf T}H^{-1}g_k.
\]
Table~\ref{tab:bb1-trajectory} lists eleven consecutive iterates.  The computation was carried out with 100-digit arithmetic and rounded for display; an upward arrow marks an increase of $F_k$ from the preceding row.

\begin{table}[htbp]
\centering
\caption{A nonmonotone BB1 trajectory and the monotone endpoint Lyapunov function.  The ratio in the last column is bounded above by $81/121\approx0.669421$ at every displayed step.}
\label{tab:bb1-trajectory}
\begingroup
\scriptsize
\setlength{\tabcolsep}{3.4pt}
\renewcommand{\arraystretch}{1.12}
\begin{tabular}{@{}rccccc@{}}
\toprule
$k$ & $\alpha_k$ & $F_k$ & $\norm{g_k}$ & $V_k$ & $V_{k+1}/V_k$\\
\midrule
0  & $0.1219512$    & $0.395122$                    & $0.981688$                 & $0.679680$                    & $0.599200$\\
1  & $0.3571429$    & $0.297661$                    & $0.776970$                 & $0.407264$                    & $0.531734$\\
2  & $0.8783784$    & $0.125892$                    & $0.554124$                 & $0.216556$                    & $3.15078\times10^{-2}$\\
3  & $0.3571429$    & $0.187852\,\uparrow$         & $1.92985$                  & $6.82321\times10^{-3}$       & $0.531734$\\
4  & $0.1000879$    & $1.23086\,\uparrow$          & $4.96021$                  & $3.62813\times10^{-3}$       & $0.229653$\\
5  & $0.1000055$    & $6.08979\times10^{-4}$       & $3.51434\times10^{-2}$    & $8.33213\times10^{-4}$       & $0.138746$\\
6  & $0.8783784$    & $4.92497\times10^{-4}$       & $3.13846\times10^{-2}$    & $1.15605\times10^{-4}$       & $3.15078\times10^{-2}$\\
7  & $0.9999999995$ & $7.28493\times10^{-6}$       & $3.81705\times10^{-3}$    & $3.64247\times10^{-6}$       & $1.99232\times10^{-17}$\\
8  & $0.9999978542$ & $1.40686\times10^{-11}$      & $1.67741\times10^{-5}$    & $7.25694\times10^{-23}$      & $7.36687\times10^{-11}$\\
9  & $0.1000000$    & $1.13955\times10^{-9}\,\uparrow$ & $1.50967\times10^{-4}$ & $5.34610\times10^{-33}$      & $2.45912\times10^{-3}$\\
10 & $--$           & $7.64263\times10^{-36}$      & $4.31747\times10^{-18}$   & $1.31467\times10^{-35}$      & $--$\\
\bottomrule
\end{tabular}
\endgroup
\end{table}
\FloatBarrier

The original objective is visibly nonmonotone: for example,
\[
  \frac{F_3}{F_2}\approx1.49217,
  \qquad
  \frac{F_4}{F_3}\approx6.55229,
  \qquad
  \frac{F_9}{F_8}\approx80.9996.
\]
At the same three transitions, the gradient norm is multiplied by approximately $3.48271$, $2.57025$, and $8.99998$, respectively.  Nevertheless, every Lyapunov ratio in the last column is below $81/121$.  The trajectory therefore exhibits repeated increases of both the original objective and the gradient norm, while $V_k$ continues to decrease at every step exactly as predicted by Theorem~\ref{thm:exact-dissipation}.

\section{Conclusion and limitations}

For each post-warm-up two-step state with two persistent spectral endpoints, the quadratic BB dynamics admit the explicit endpoint Lyapunov function
\[
  \overline f(x)
  =
  \norm{P_a\nabla f(x)}^{\frac{2b}{a+b}}
  \norm{P_b\nabla f(x)}^{\frac{2a}{a+b}}.
\]
It satisfies the exact dissipation law
\[
  \overline f(x_{k+1})
  =
  \left(\frac{b-a}{b+a}\right)^2
  e^{-2D_{a,b}(\tau_k)}
  \overline f(x_k),
  \qquad D_{a,b}(\tau_k)\ge0,
\]
and the endpoint exponents are the unique minimax weights among endpoint monomials.  De-normalization of the coboundary certificate in \cite{YangYuan2026} gives the function directly.

The same Lyapunov law yields a proof of $R$-linear convergence.  If the full gradient norm has not contracted, the normalized endpoint product tends to zero.  Induction over the finite index set and compactness give a uniform finite-block contraction, and hence $R$-linear decay of the gradient and error.  Fixed positive weighted delayed rules follow by spectral conjugacy.

The block constants obtained by compactness are nonconstructive and do not recover the sharp factor $(b-a)/(b+a)$ in \cite{YangYuan2026}, and the face induction is specific to finite spectra; continuous-spectrum and nonlinear extensions require different arguments.

\section*{Acknowledgements}
Shutai Yang thanks Shixiang Chen for supervising his undergraduate thesis at the University of Science and Technology of China, and Xiaowei Xu for supervising his project under the National College Students Innovation and Entrepreneurship Training Program at the same university.  The line of research from which the present article emerged grew out of those two undergraduate projects.

\section*{Statements and Declarations}

\subsection*{Funding}
The work of Shutai Yang was supported by the National College Students Innovation and Entrepreneurship Training Program, administered by the University of Science and Technology of China (Project No.~202510358091).  The work of Ya-xiang Yuan was supported by the National Natural Science Foundation of China (Grant No.~12288201).

\subsection*{Competing interests}
The authors have no relevant financial or non-financial interests to disclose.

\subsection*{Author contributions}
Shutai Yang conceived the study, developed the mathematical results, carried out the analysis, and wrote the original draft.  Ya-xiang Yuan provided guidance on the research direction and reviewed the manuscript.  Both authors read and approved the final manuscript.

\subsection*{Declaration on the use of generative AI}
During the preparation of this manuscript, Shutai Yang used OpenAI language models for language and \LaTeX{} editing, and the presentation of mathematical notation.  He independently checked and revised all model-assisted material.  Both authors reviewed and approved the final manuscript and take full responsibility for its content.

\subsection*{Data availability}
No datasets were generated or analyzed during the current study.

\subsection*{Code availability}
No research software or reusable computational code is associated with this theoretical note.  The numerical values in Section~\ref{sec:numerics} are direct evaluations of the displayed two-dimensional recurrence.

\end{document}